\documentclass[12pt,reqno]{amsart}
\usepackage{amsmath,amssymb,amsthm,bbm,fullpage,cite,verbatim}
\usepackage{amssymb}
\usepackage[utf8]{inputenc}
\usepackage{bbm}

\DeclareMathOperator*{\E}{\mathbb{E}}
\DeclareMathOperator*{\Prob}{\mathbb{P}}

\newtheorem{theorem}{Theorem}
\newtheorem{lemma}[theorem]{Lemma}

\newtheorem{corollary}[theorem]{Corollary}

\theoremstyle{definition}

\newcommand\eps{\varepsilon}

\renewcommand\leq{\leqslant}
\renewcommand\geq{\geqslant}
\renewcommand\le{\leqslant}
\renewcommand\ge{\geqslant}

\begin{document}

\title{Large Sumsets from Medium-Sized Subsets}
\author{B\'ela Bollob\'as \and Imre Leader \and Marius Tiba}

\address{Department of Pure Mathematics and Mathematical Statistics,
Wilberforce Road,
Cambridge, CB3 0WA, UK, and Department of Mathematical Sciences,
University of Memphis, Memphis, TN 38152, USA}\email{b.bollobas@dpmms.cam.ac.uk}

\address{Department of Pure Mathematics and Mathematical Statistics,
Wilberforce Road, Cambridge, CB3 0WA, UK}\email{i.leader@dpmms.cam.ac.uk}

\address{IMPA, Estrada Dona Castorina 110, Rio de Janeiro, Brazil}\email{sirmariustiba@gmail.com}

\thanks{The first author was partially supported by NSF grant DMS-1855745}

\begin{abstract}
The classical Cauchy--Davenport inequality gives a lower bound for the size of
the sum of two subsets of ${\mathbb Z}_p$, where $p$ is a prime. Our main aim in this paper is to
prove a considerable strengthening of this inequality, where we take only a small number of points from each of
the two subsets when forming the sum. One of our results is that
there is an absolute constant $c>0$ such that if $A$ and $B$ are subsets of ${\mathbb Z}_p$ with
$|A|=|B|=n\le p/3$ then there are subsets $A'\subset A$ and
$B'\subset B$ with $|A'|=|B'|\le c \sqrt{n}$ such that $|A'+B'|\ge 2n-1$. In fact, we show that one
may take any sizes one likes: as long as $c_1$ and $c_2$ satisfy $c_1c_2 \ge cn$ then we may choose
$|A'|=c_1$ and $|B'|=c_2$. We prove related
results for general abelian groups.
\end{abstract}

\maketitle

\section{Introduction}

The Cauchy--Davenport theorem~\cite{Cauchy, Dav1, Dav2} states that if $p$ is a prime and $A$ and $B$ are non-empty
subsets of ${\mathbb Z}_p$ with $|A|+|B|\le p+1$ then $|A+B|\ge |A|+|B|-1$. Intervals show that this bound is best possible.
Over the years, this classical result
was followed by a host of important contributions about sums of subsets of groups, including other abelian groups such as
 ${\mathbb Z}$ itself. For these
 contributions, see, among others,
 Mann~\cite{Mann1, Mann2}, Kneser~\cite{Knes},
 Erd\H{o}s and Heilbronn~\cite{ErdHeil}, Freiman~\cite{Frei-59, Frei-62, Frei-book, Frei-87}, Pl\"unnecke~\cite{Plu-70},
 Ruzsa~\cite{Ruz-89}, Dias da Silva and Hamidoune~\cite{DiHa}, Alon, Nathanson and Ruzsa~\cite{AlNaRu},
Shao~\cite{shao}, Stanchescu~\cite{Stan}, Breuillard, Green and Tao~\cite{BGT-doubling}, as well as the books of
Nathanson~\cite{Nathbook}, Tao and Vu~\cite{taobook} and Grynkiewicz~\cite{grynkiewicz-2}.

Recently, a new direction of research was started in ~\cite{BLT}: can we get similar bounds for the size of the sum
if $A+B$ is replaced by $A+B'$, where $B'$ is a {\em small} subset of $B$? Among other results, it was proved that if
$A$ and $B$ are finite non-empty subsets of ${\mathbb Z}$ with $|A| \ge |B|$ then $B'$ can be taken to be really small:
there are {\em three} elements $b_1, b_2, b_3 \in B$ such that $|A+\{b_1,b_2,b_3\}| \ge |A|+|B|-1$. For
${\mathbb Z}_p$, it was shown that if
$A, B \subset {\mathbb Z}_p$ with $|A|=|B|=n\le p/3$ then $B$ has a subset $B'$ with at most $c$ elements
such that $|A+B'|\ge 2n-1$, where $c$ is an absolute constant. (Here again $p$ is prime, and for the rest of this paper $p$ will always denote a prime.)

Our aim in this paper is to prove that actually one can replace {\em both} $A$ and $B$ by appropriate small subsets.
In the result just mentioned, the product of the sizes of our two subsets in the sum, namely $A$ and $B'$, is $cn$,
and trivially we cannot ever get a sum of size linear in $n$ without the product of the sizes of the two sets being
linear in $n$. But, remarkably, one can indeed always choose subsets $A'$ of $A$ and $B'$ of $B$, of any desired
sizes, as long as the product of these sizes is linear in $n$.

The result mentioned in the Abstract has both sizes being a constant times $\sqrt{n}$, and the sets $A$ and $B$
themselves have size bounded away from $p/2$. The general form of our result is as follows.

\begin{theorem}\label{all_thm1}
  For all $\alpha, \beta>0$ there exists $c>0$ such that the following holds.
  Let $A$ and $B$ be non-empty subsets of $\mathbb{Z}_p$ with
  $\alpha |B| \leq |A| \leq \alpha^{-1} |B|$ and $|A|+|B| \leq (1-\beta)p$.
  Then, for any integers $1 \leq c_1 \leq |A|$ and $1 \leq c_2 \leq |B|$ such that $c_1c_2 \geq c\max(|A|,|B|)$,
  there exist subsets $A' \subset A$ and $B' \subset B$, of sizes $c_1$ and $c_2$ such that
  $|A'+B'| \geq |A|+|B|-1$.
\end{theorem}

It is rather surprising that this holds with no other restrictions on the values of $c_1$ and $c_2$, whatever the
form of the sets $A$ and $B$. We remark that the constant $c$ does have to depend on $\beta$: if our two sets
are allowed to have sizes whose sum approaches $p$ then taking random subsets shows that $c$ has to grow. Also, it is
easy to see that $c$ depends on $\alpha$. For example, if $B$ is far larger than $A$ then trivially taking $c$ points
from $B$ to be summed with all of $A$ will yield a sumset that is too small.

Our main tool is a similar result that is valid in general abelian groups. The reader will note that it is `worse'
than the above result in that there is an error term, but it is also far `better' because the lower bound comes from
the sum of two actual sets (the sets $A^*$ and $B^*$ below) rather than merely from the lower bound that comes from the
sum of their sizes (or, more precisely, the general Kneser lower bound that generalises the Cauchy--Davenport theorem).

\begin{theorem}\label{implication_shao1-prov} For all $K$ and $\varepsilon >0$ there exists $c$ such that the
  following holds.
  Let $A$ and $B$ be finite non-empty subsets of an abelian group, and let $1 \leq c_1 \leq |A|$ and
  $1 \leq c_2 \leq |B|$ be integers satisfying $c_1c_2 \geq c \max(|A|,|B|)$. Then there are subsets $A^* \subset A$
and $B^* \subset B$, with $|A^*|\ge (1-\varepsilon)|A|$ and
$|B^*|\ge (1-\varepsilon)|B|$, such that
if we select points $a_1, \hdots, a_{c_1}$ and $b_1,\hdots,b_{c_2}$ uniformly at random from $A^*$ and
$B^*$ then, writing $A'$ for $\{a_1, \hdots, a_{c_1}\}$ and $B'$ for $\{b_1, \hdots, b_{c_2}\}$, we have
$${\mathbb E}|A'+ B'|
 \geq \min\big( (1-\varepsilon)|A^*+ B^*|,\ K |A^*|,\ K |B^*| \big).$$
\end{theorem}

We remark that the terms $K |A^*|$ and $K |B^*|$ are only present to deal with unimportant cases: the key term is
$(1-\varepsilon)|A^*+ B^*|$. Thus the result is informally somehow saying that, in terms of sumsets, $A^*$ and $B^*$ may really be approximated by very small subsets of themselves (and indeed most subsets will do).

Interestingly, while this theorem is for general abelian groups, Theorem~\ref{all_thm1} is only about $\mathbb{Z}_p$. The passage
between these does require quite a lot of work.


The plan of the paper is as follows. In Section 2 we give various prerequisites that we shall need. Then in Section
3 we prove Theorem~\ref{implication_shao1-prov}, and also provide the consequence of it in $\mathbb{Z}_p$ that we use in
the proof of Theorem~\ref{all_thm1}.
In Section 4 we prove Theorem~\ref{all_thm1},
and the last section, Section 5, contains open problems.

\vspace{5pt}
Our notation is standard.
Sometimes we write `$x \mod d$' as shorthand for the infinite arithmetic progression
$\{ y \in \mathbb{Z}: y \equiv x \mod d \}$, and refer to it as a {\em fibre} mod $d$.
When $S$ is a subset of $\mathbb{Z}$ we often write $S^x$ for the intersection of this fibre with
$S$ -- when the value of $d$ is clear. (We sometimes write $S^x$ as $S^x_d$ when we want to stress the value of $d$.)
Thus $S^x=S\cap \pi^{-1}(x)$, where $\pi = \pi_d$ denotes
the natural projection from $\mathbb{Z}$ to $\mathbb{Z}_d$.

\vspace{5pt}
When we write a probability or an expectation over a finite set, we always assume that the elements
of the set are being sampled uniformly. Thus, for example,
for a finite set $X$ we denote the expectation and probability when we sample uniformly over
all $x \in X$ by  ${\mathbb E}_{x\in X} \text{ and } \Prob_{x\in X}$.
We also often sample uniformly over all $c$-sets of a given set $X$. In most of those cases,
we could instead sample $c$ elements uniformly and independently, but the notation would tend to get
unwieldy, and this is why we use the sampling over all $c$-sets instead.

\vspace{10pt}

Before we turn to the next section, let us draw attention to the superficial similarity of
our problems to a beautiful result of Ellenberg~\cite{Ell}. Given a prime $p$ and a positive
integer 
$d$, let $f(p^d)$ be the smallest integer such that for any sets $S, T \subset {\mathbb Z}_p^d$
there are subsets $S' \subset S$ and $T'\subset T$ satisfying $(S+T')\cup (S'+T)=S+T$
and $|S'|+|T'| \le f(p^d)$. Ellenberg proved that $f(p^d) \le (cp)^d$, where $c<1$ is an absolute
constant.

\vspace{10pt}

\section{Prerequisites}
In this section we collect together the various prerequisites that we will need. Each of these may be treated
as a `black box': knowledge of their proofs will not be required.

The first of the three theorems we shall need is due to Shao~\cite{shao}, and  concerns
restricted sums. Let $A$ and $B$ be subsets of an abelian group, and let
$\Gamma \subset A \times B$. The {\em $\Gamma$-restricted sum} of $A$ and $B$ is
$A+_{\Gamma}B=\{a+b: \ a\in A,\ b\in B, (a,b) \in \Gamma \}$. Here is the result of Shao.
\begin{theorem}\label{thm_shao}
For all $\varepsilon, K>0$ there exists $\delta>0$ such that the following holds. Let $G$ be an
abelian group and let $N \in \mathbb{N}$. Let $A,B \subset G$
be two subsets with $|A|, |B| \geq N$. Let $\Gamma \subset A \times B$ be a subset with
$|\Gamma|\geq (1-\delta)|A||B|$. If $|A+_{\Gamma}B| \leq KN$, then there
exist $A_0 \subset A$ and $B_0 \subset B$ such that
\[
|A_0| \geq (1-\varepsilon)|A|
\text{ and } |B_0|\geq (1-\varepsilon)|B| \text{ and } |A_0+B_0| \leq |A+_{\Gamma}B|+\varepsilon N.
\]
\end{theorem}
The second theorem is an easy corollary of a theorem of Grynkiewicz~\cite{grynkiewicz-2}.
\begin{theorem}\label{Freiman_Zp}
Given $\beta, \gamma >0$ there is an $\varepsilon>0$ such that the following holds.
  Let $A$ and $B$ be subsets of $\mathbb{Z}_p$. Suppose that
$2\leq \min(|A|,|B|)\text{ and }|A|+|B| \leq (1-\beta)p$
and $|A+B| \leq |A|+|B|-1 + \varepsilon \min(|A|,|B|)$.
Then there are arithmetic progressions $P$ and $Q$ with the same common difference that contain $A$ and $B$ and satisfy $|P \Delta A|\leq  \gamma \min(|A|,|B|)  \text{ and }  |Q \Delta B|\leq \gamma \min(|A|,|B|)$.
\end{theorem}

The last theorem we need is a somewhat technical result from \cite{BLT}. It gives a strengthening of the result from \cite{BLT}
mentioned above about sums in $\mathbb{Z}_p$, when the sets $A$ and $B$ `relate nicely' to intervals.

\begin{theorem}\label{CD_technical}
For all $\beta>0$ there exists $\gamma>0$ such that for every $\alpha>0$ there is a value of $c$ for which the
following
holds. Let $A$ and $B$ be subsets of $\mathbb{Z}_p$ and let  $I= [p_l,  p_r]$ and $J= [q_l,  q_r]$
be intervals in $\mathbb{Z}_p$ satisfying $\alpha |J| \leq |I| \leq \alpha^{-1} |J|$,\  $|I|+|J|\leq (1-\beta)p$,
$\max(|A\Delta I|, |B\Delta J|) \leq \gamma \min (|I|, |J|)$\ and\  $\{q_l, q_r\} \subset B \subset J$.
Then there is a family $\mathcal{F}\subset B^{(c)}$,\ depending only on $I$, $J$ and $B$ (but not on $A$),\ such that
\[
\E_{B'\in \mathcal{F}}|A+B'|\geq |A|+|J|-1 \geq |A|+|B|-1.
\]
\end{theorem}

\section{Proof of Theorem 2}

In this section we prove our main result on general abelian groups, Theorem~\ref{implication_shao1-prov}.

\vspace{5pt}
We start by giving a brief overview of the proof. Although this paper is self-contained, we mention that the
reader who is familiar with \cite{BLT} will see that this proof is similar in spirit to the proof of Theorem 10 in
that paper.

We will repeatedly apply Theorem~\ref{thm_shao} in order to construct a decreasing sequence
of $s + 1$ pairs of sets $(A, B) = (A_0, B_0), (A_1, B_1), \dots , (A_s, B_s)$,
satisfying $A_i \subset  A_{i-1}$ and $B_i
\subset  B_{i-1}$ and $|A_i| >
(1 -
\varepsilon / s)^i |A|$ and $|B_i| > (1 - \varepsilon / s)^i |B|$.
Having constructed $A_i$ and $B_i$ we divide the elements of  $A_i+B_i$ into the
set $P_i$ of `popular' ones (those hit at least $\alpha \min(|A_i|,|B_i|)$ times)  and
the set $U_i$ of unpopular ones. And we let $\Gamma$ be the pairs summing to
popular elements.

We first deal with the situation when $|\Gamma|$ is much smaller than $|A_i| |B_i|$ or
$|P_i|$ is much larger than $K \min(|A_i|,|B_i|)$. In both cases a simple
computational check shows that the pair $(A_i, B_i)$ has the desired
properties. If we are not in this situation then we can apply Thm~\ref{thm_shao} to $A_i, B_i$ and $\Gamma$ to construct the sets
$A_{i+1}$, $B_{i+1}$. These satisfy $|A_{i+1} +
B_{i+1}| < |P_i| + (\varepsilon / s) \min(|A_{i+1}|, |B_{i+1}|)$.

We then deal with the case when the process continues for at least $s$
steps. In this case we have  $|A_{i+1} + B_{i+1}| -  (\varepsilon / s)
\min(|A_{i+1}|, |B_ {i+1}|) < |P_i| < |A_i+B_i|$, where the second inequality is
clear and the first is by Theorem~\ref{thm_shao}.

As $|P_0| < 10 K \min (|A_0|, |B_0|)$ and $|A_s + B_s| > |B_s| > |B_0|/2$, we
deduce that there exists $i$ such that $|P_i|$ is about 
then easy to check that the pair of sets $(A_i, B_i)$ has the desired
property. Indeed, if $c$ is large enough (depending on $\alpha$),
then $|A'+B'|$ is about $|P_i|$ as each point in $P_i$ is hit with high probability. So
we conclude that $|A'+B'|$ is about $|A_i+B_i|$.

\vspace{5pt}
We now turn to the proof itself.

\begin{proof}[Proof of Theorem \ref{implication_shao1-prov}]
Fix $\varepsilon>0$  and $K>0$, where we assume that $\varepsilon$ is sufficiently small and
$K$ is sufficiently large.
Pick $s=\lfloor \frac{50K}{\varepsilon} \rfloor$. Let
$\delta$ be given by Theorem~\ref{thm_shao} with parameters $\frac{\varepsilon}{s}$ and $10K$.
Also pick  $\alpha = \delta/16 K$. Finally, pick $c\geq \max(2^{10}K/ \delta, 2^{10}  |\log(\epsilon)|/ \alpha)$. We may assume
that $|A| \geq |B|$.

We shall examine a process in which we repeatedly apply Theorem \ref{thm_shao} in order to construct a
decreasing sequence of $s+1$ pairs of sets $(A,B)=(A_0,B_0),(A_1,B_1), \hdots ,(A_s,B_s),$
satisfying $A_i \subset A_{i-1} \text{ and } B_i \subset B_{i-1}$ and
$|A_i|\geq (1-\varepsilon/s)^i |A|$ and  $|B_i|\geq (1-\varepsilon/s)^i |B| $.
Fix $i<s$ and assume that the pair of sets $(A_i,B_i)$ has already been
constructed. We shall either stop the process at step $i$ or construct the pair of
sets $(A_{i+1},B_{i+1})$.

Let $A_i+B_i=C_i^+\sqcup C_i^-$ be the partition into `popular'  and `unpopular' elements  given by
$ C_i^+=\{c \in A_i+B_i: \ |(c-A_i)\cap B_i| \geq \alpha |B_i|\}$, and
$C_i^-=\{c \in A_i+B_i: \ |(c-A_i)\cap B_i| < \alpha |B_i|\}$.
Also, let the partition $A_i \times B_i=\Gamma_i\sqcup \Gamma_i^c$ be
given by $\Gamma_i=\{(a,b) \in A_i \times B_i: \ a+b \in C_i^+  \} \subset A_i \times B_i$, and $\Gamma_i^c=\{(a,b) \in A_i
\times B_i : \ a+b \in C_i^-  \} \subset A_i \times B_i$,
so that $A_i+_{\Gamma_i}B_i=C_i^+$  and  $A_i+_{\Gamma^c_i}B_i=C_i^-$.
Finally, for
each $x \in A_i+B_i$ set
$$
A_i^x=(x-B_i)\cap A_i \text{ and } B_i^x=(x-A_i)\cap B_i \text{ such that } A_i^x=x-B_i^x,$$
$$
r_i(x)=|A_i^x|=|B_i^x|=|\{(a,b)\in A_i \times B_i: \
x=a+b\}|,
$$
so that $ \sum_x r_i(x)=|A_i||B_i|.$
We stop this process `early', at step $i$, if $$|\Gamma_i|<(1-\delta)|A_i||B_i| \text{ or } |A_i+_{\Gamma_i}B_i|> 10K \min(|A_i|,|B_i|).$$
Otherwise, we apply Theorem~\ref{thm_shao} with parameters
$\varepsilon/s, 10K$ to the pair of sets $(A_i,B_i)$.  Thus we produce a pair of sets $(A_{i+1},B_{i+1})$, satisfying
$A_{i+1} \subset A_i$,  $B_{i+1} \subset B_i$, \
$|A_{i+1}| \geq (1-\varepsilon/s)|A_i|$, \  $|B_{i+1}|\geq (1-\varepsilon/s)|B_i|$
and $|A_{i+1}+B_{i+1}| \leq |A_i+_{\Gamma_i}B_i|+\frac{\varepsilon}{s} \min(|A_i|,|B_i|)$.
We shall analyse separately the cases in which the process continues until the end and in
which the process stops before that.

Before we begin, we need one easy estimate. Suppose the process continues until step $j$. If we
choose elements $a_1, \hdots, a_{n_A}$ and $b_1, \hdots, b_{n_B}$ uniformly at random
from $A_j$ and $B_j$, and we write $A'_j=\{a_1, \hdots, a_{n_A}\}$ and $B'_j=\{b_1, \hdots, b_{n_B}\}$, then we have
\begin{equation}\label{unification1}
\begin{split}
    {\mathbb E}|A_j'+B_j'|
    &= \sum_x \Prob(x \in  A_j'+B_j')\\
    &\geq \sum_{x} \sum_{\substack{ X\subset A_j^x, Y\subset B_j^x \\ X=x-Y \\ |X|=|Y|> \frac{n_Br_j(x)}{2|B_j|} }}
    \mathbb{P} (B_j' \cap B_j^x = Y \text{ and } A_j' \cap X \neq \emptyset)\\
    &= \sum_{x} \sum_{\substack{ X\subset A_j^x, Y\subset B_j^x \\ X=x-Y \\ |X|=|Y|> \frac{n_Br_j(x)}{2|B_j|} }}
    \mathbb{P} (B_j' \cap B_j^x= Y )  \mathbb{P}(A_j' \cap X \neq \emptyset)\\
    &= \sum_{x} \Prob(|B_j'\cap B_j^x| >  \frac{n_Br_j(x)}{2|B_j|}) \min_{\substack{X \subset A_j^x \\ |X| >
        \frac{n_Br_j(x)}{2|B_j|}}}\Prob(|A_j' \cap X|  >  0)\\
      &\geq  \sum_{x} (1-\exp(- \frac{n_Br_j(x)}{16|B_j|} )) (1-\exp(- \max(\frac{n_A}{4|A_j|},\frac{n_An_Br_j(x)}{8|A_j||B_j|}))).
\end{split}
\end{equation}
Here, the last inequality follows from Chernoff's inequality (see for example Corollary 1.9 in \cite{taobook}) and the fact that $|X| > \frac{n_Br_j(x)}{2|B_j|} $ is equivalent to $|X| \geq \max(1, \frac{n_Br_j(x)}{2|B_j|})$.


\vspace{7pt}
{\bf Claim A.}
{\em
Suppose the process
stops early, say at step $j<s$. Then the pair of sets $(A_j,B_j)$ has the desired properties.
}

\begin{proof}
\vspace{7pt}
\noindent
\textbf{Case 1:} Consider first the
case when $|C^+_j|=|A_j+_{\Gamma_j}B_j|> 10K \min(|A_j|,|B_j|).$ For $x\in C_j^+$,
by construction we have that $r_j(x)=|(x-A_j)\cap B_j|\geq \alpha
|B_j|.$ If we choose elements $a_1, \hdots, a_{n_A}$ and $b_1, \hdots, b_{n_B}$ uniformly at random
from $A_j$ and $B_j$, and we write $A'_j=\{a_1, \hdots, a_{n_A}\}$ and $B'_j=\{b_1, \hdots, b_{n_B}\}$, then we have
\begin{eqnarray*}
{\mathbb E}|A_j'+B_j'|
    &\geq&  \sum_{x} (1-\exp(-\frac{n_Br_j(x)}{16|B_j|} )) (1-\exp(- \max(\frac{n_A}{4|A_j|},\frac{n_An_Br_j(x)}{8|A_j||B_j|}))) \\
    &\geq& \sum_{x}(1-\exp(- \frac{cr_j(x)}{16|B_j|}))^2
      \geq \sum_{x \in C^+_j} (1-\exp(- \frac{cr_j(x)}{16|B_j|}))^2\\
      &\geq& \sum_{x \in C^+_j}(1- \exp(-2^{-4} \alpha c))^2
      \geq |C_j^+|/2
    \geq 5K \min(|A_j|,|B_j|).
\end{eqnarray*}
Here the first inequality follows from \eqref{unification1}; the second from the hypothesis $n_An_B \geq c|A|\geq c|A_j|$
which, in particular, gives $n_B \geq c|B_j|$; the fourth from the construction as $r_j(x) \geq \alpha |B_j|$
for $x \in C_j^+$; the fifth from the hypothesis $c \geq 2^{10}/ \alpha$; and the last inequality follows from
the assumption on the size of $C^+_j$.

\vspace{7pt}
\noindent
\textbf{Case 2:} Consider now the case when $|\Gamma_j|<(1-\delta)|A_j||B_j|.$ By construction,
$\sum_{x\in C_j^-}r_j(x) \geq \delta|A_j||B_j|.$ Moreover, for
$x\in C_j^-$ we have $r_j(x)=|(x-A_j)\cap B_j|\leq \alpha
|B_j|.$
If we choose elements $a_1, \hdots, a_{n_A}$ and $b_1, \hdots, b_{n_B}$ uniformly at random
from $A_j$ and $B_j$, and we write $A'_j=\{a_1, \hdots, a_{n_A}\}$ and $B'_j=\{b_1, \hdots, b_{n_B}\}$, then we have the following sequence of inequalities. To make the formulae less cluttered, we define $D_j^-=\{x: n_B r_j(x)\le 2|B_j|\}$ and $D_j^+=\{x: n_B r_j(x)> 2|B_j|\}$.

\begin{eqnarray*}
    {\mathbb E}|A_j'+B_j'|\hspace{-8pt}
    &\geq& \hspace{-8pt} \sum_{x} (1-\exp(- \frac{n_Br_j(x)}{16|B_j|} )) (1-\exp(- \max(\frac{n_A}{4|A_j|},\frac{n_An_Br_j(x)}{8|A_j||B_j|}))) \\
    &\geq& \hspace{-12pt} \sum_{x \in D_j^-} (1-\exp(- \frac{n_Br_j(x)}{16|B_j|} ))
    (1-\exp(- \frac{n_A}{4|A_j|})) + \sum_{x\in D_j^+} 2^{-1}(1-\exp(- \frac{n_An_Br_j(x)}{8|A_j||B_j|}))\\
    &\geq& \hspace{-12pt} \sum_{x\in D_j^-} (1-\exp(- \frac{n_Br_j(x)}{16|B_j|} ))
    (1-\exp(- \frac{n_A}{4|A_j|})) + \sum_{x\in D_j^+} 2^{-1}(1-\exp(- \frac{8Kr_j(x)}{\delta|B_j|}))\\
    &\geq& \hspace{-19pt} \sum_{x \in C_j^-\cap D_j^-} \frac{n_Br_j(x)}{32|B_j|} \frac{n_A}{8|A_j|}
        + \sum_{x \in C_j^- \cap D_j^+} \frac{Kr_j(x)}{\delta |B_j|}
    \geq \sum_{x \in C_j^-} \frac{K r_j(x)}{\delta |B_j|} \geq K|A_j|.
\end{eqnarray*}
Here, the first inequality follows from \eqref{unification1}, the second inequality follows by splitting
into two cases according to how $\frac{n_B r(x)}{2|B_j|}$  compares to 1, the third inequality follows
from the hypothesis $n_An_B \geq c|A| \geq c|A_j|$ and $c \geq 2^{10}K/\delta$, the fourth inequality
follows from the two facts that  $1-\exp(-t) \geq t/2$ for $0 \leq t \leq 1/2 $ and $\frac{8Kr_j(x)}{\delta |B_j|}
\leq \frac{8K \alpha}{\delta} <1/2 $ for $x \in C_j^-$, the fifth inequality follows again from the
hypothesis $n_An_B \geq c |A| \geq c |A_j| $ and $c\geq  2^{10} K/ \delta $, and the last inequality
follows from the original assumption that $\sum_{x \in C_j^-} r_j(x) \geq \delta |A_j||B_j|$.

We conclude the pair of sets $(|A_j|,|B_j|)$ has the desired properties, so Claim A is proved.
\end{proof}

We now turn to the case when the process does not stop early.

\vspace{7pt}
{\bf Claim B.} {\em
Suppose that the process continues until the terminal step $s$.
Then there is an index $j\leq s$ such that the  pair of sets $(A_j,B_j)$ has the desired
properties.}
\begin{proof}
Note that $|A_{1}+B_{1}| \leq |A_0+_{\Gamma_0}B_0|+\frac{\varepsilon}{s} \min(|A_0|,|B_0|) \leq 11K \min(|A_0|,|B_0|).$
Moreover, $11K \min(|A|,|B|)= 11K \min(|A_0|,|B_0|)\geq |A_{1}+B_{1}| \geq \hdots \geq |A_{s}+B_{s}| \geq 0.$ Therefore
$|A_{j+1}+B_{j+1}|\geq |A_{j}+B_{j}|-\frac{11K}{s-1}\min(|A|,|B|)$ for some index $j$,
$1 \leq j \leq s$.  We shall
show that the pair  $(A_j,B_j)$ has the desired properties.
Indeed, by construction,
\[
|A_{j}+_{\Gamma_j}B_{j}|+\frac{\varepsilon}{s}
\min(|A|,|B|) \geq |A_{j}+_{\Gamma_j}B_{j}|+\frac{\varepsilon}{s}
\min(|A_j|,|B_j|) \geq |A_{j+1}+B_{j+1}|.
\]
It follows that
\[
  |A_j+_{\Gamma_j}B_j|\geq |A_{j}+B_{j}|- \big(\frac{11K}{s-1}+\frac{\varepsilon}{s} \big)\min(|A|,|B|)
  \geq \big(1- \frac{20K}{s} \big)|A_{j}+B_{j}|.
\]
If we choose elements $a_1, \hdots, a_{n_A}$ and $b_1, \hdots, b_{n_B}$ uniformly at random
from $A_j$ and $B_j$, and we write $A'_j=\{a_1, \hdots, a_{n_A}\}$ and $B'_j=\{b_1, \hdots, b_{n_B}\}$, then we have
\begin{eqnarray*}
    {\mathbb E}|A_j'+B_j'|
    &\geq&  \sum_{x} (1-\exp(-  \frac{n_Br_j(x)}{16|B_j|} )) (1-\exp(- \max(\frac{n_A}{4|A_j|},\frac{n_An_Br_j(x)}{8|A_j||B_j|})))\\
    &\geq& \sum_{x}(1-\exp(- \frac{cr_j(x)}{16|B_j|}))^2
      \geq \sum_{x \in C^+_j} (1-\exp(- \frac{cr_j(x)}{16|B_j|}))^2\\
      &\geq& \sum_{x \in C^+_j}(1- \exp(-2^{-4} \alpha c))^2
      \geq (1- \exp(-2^{-4} \alpha c))^2 |A_j+_{\Gamma_j}B_j|\\
      &\geq& (1- \exp(-2^{-4} \alpha c))^2 (1-\frac{20 K}{s}) |A_j+B_j|
       \geq (1-\epsilon) |A_j+B_j|.
\end{eqnarray*}
Here the first inequality follows from \eqref{unification1}; the second from the hypothesis
$n_An_B \geq c|A|\geq c|A_j|$ which, in particular, gives $n_B \geq c|B_j|$; the fourth inequality
holds by the construction, as $r_j(x) \geq \alpha |B_j|$ for $x \in C_j^+$; the fifth inequality follows from the
hypothesis $c \geq 2^{10}/ \alpha |\log(\epsilon)|$; and the last inequality follows from the assumption on the size of $C^+_j$.

Thus the pair $(|A_j|,|B_j|)$ has the desired properties, so Claim B is proved.
\end{proof}

\vspace{5pt}
This concludes the proof of Theorem~\ref{implication_shao1-prov}: whether the process stops early
or does not, the  pair of sets $(A_j,B_j)$ has the desired
properties.
\end{proof}

When we come to proving Theorem~\ref{all_thm1}, we shall need the following consequence of this.

\begin{theorem}\label{all_thm2}
  For all $\beta, \gamma>0$ there exists $\epsilon>0$ such that for all $\alpha>0$ there is a value of $c$
  for which the following holds.  Let $A$ and $B
  $ be subsets of $\mathbb{Z}_p$ and let $1 \leq c_1 \leq |A|$ and $1 \leq c_2 \leq |B|$ be integers such
  that $c_1c_2 \geq c\max(|A|,|B|)$. Suppose that
\begin{equation}\label{all_eq2}
    2 \leq \min(|A|,|B|),  \alpha |B| \leq |A| \leq \alpha^{-1} |B| \text{ and } |A|+|B| \leq (1-\beta)p
\end{equation}
and
\begin{equation}\label{all_eq3}
    \E_{A' \in A^{(c_1)},\  B' \in B^{(c_2)}} |A'+B'| \leq |A|+|B|-1 +\eps \min(|A|,|B|).
\end{equation}
Then there exist arithmetic progressions $P$ and $Q$ with the same common difference and
\begin{equation*}
    \max(|A\Delta P|, |B \Delta Q|) \leq \gamma \min(|A|,|B|).
\end{equation*}
\end{theorem}

\begin{proof}
We start by fixing some parameters. Fix $2^{-10}>\beta>\gamma>\alpha>0$. Let $\varepsilon_1$ be the output of Theorem~\ref{Freiman_Zp} with
  input $\beta,2^{-11}\gamma$. Let
  $\varepsilon=\min(2^{-2}\varepsilon_1,2^{-6}\gamma)$. Let $\mu=2^{-12}\min(\alpha \varepsilon, \alpha \gamma)$.
  Finally, let $c$ be the output of Theorem~\ref{implication_shao1-prov} with input $(4/\alpha,\mu)$.

By Theorem~\ref{implication_shao1-prov}, there are subsets $A^*\subset A$ and $B^*\subset B$ with
\begin{equation}\label{eq22.1*}
    |A^*|\geq \lceil (1-\mu)|A|\rceil  \text{ and } |B^*|\geq \lceil (1-\mu)|B|\rceil
\end{equation}
such that
\begin{equation}\label{eq22.2*}
  \E_{A' \in A^{(c_1)}, \ B'\in B^{(c_2)}}| A'+B'| \geq \min \bigg((1-\mu)|A^*+B^*|\text{, } \frac{4}{\alpha}|A|\text{, }
  \frac{4}{\alpha}|B|\bigg).
\end{equation}
By \eqref{all_eq2} we have
\begin{equation}\label{eq22.25*}
\min\bigg(\frac{4}{\alpha}|A|,\frac{4}{\alpha}|B|\bigg) \geq 4\max(|A|,|B|) >
    |A|+|B|-1+\varepsilon\min(|A|,|B|),
\end{equation}
and by \eqref{all_eq2} and \eqref{eq22.1*} we find that
\begin{equation}\label{eq22.27*}
  2\leq \min\big(|A^*|,|B^*|\big), \ \ \ \frac{\alpha}{2}
  |B^*| \leq |A^*| \leq \frac{2}{\alpha}|B^*| \ \ \ {\rm and} \ \ \  |A^*|+|B^*| \leq (1-\beta)p.
\end{equation}
Hence \eqref{all_eq3}, \eqref{eq22.2*} and \eqref{eq22.25*} imply
\begin{equation*}
    |A|+|B|-1+\varepsilon\min(|A|,|B|) \geq (1-\mu)|A^*+B^*|.
\end{equation*}
Combining this with \eqref{eq22.1*} and \eqref{eq22.27*}, we find that
\begin{equation}\label{eq22.3*}
    \begin{split}
        |A^*+B^*| &\leq (1-\mu)^{-1}(|A|+|B|-1+\varepsilon\min(|A|,|B|)) \\
        &\leq (1-\mu)^{-2}|A^*|+(1-\mu)^{-2}|B^*|-1+2\varepsilon \min(|A^*|, |B^*|)\\
        &\leq |A^*|+|B^*|-1+8\mu\max(|A^*|,|B^*|)+ 2\varepsilon\min(|A^*|,|B^*|)\\
        &\leq |A^*|+|B^*|-1+4\varepsilon\min(|A^*|,|B^*|).
        \end{split}
\end{equation}
Recalling Theorem~\ref{Freiman_Zp} and \eqref{eq22.27*}, we see that there
exist arithmetic progressions $P$ and $Q$ with
the same common difference such that
\begin{equation*}
    |A^*\Delta P| \leq \frac{\gamma}{2^{11}}\min(|A^*|,|B^*|) \text{ and } |B^*\Delta Q| \leq \frac{\gamma}{2^{11}}\min(|A^*|,|B^*|)
\end{equation*}

The conclusion now follows from this and \eqref{eq22.1*}:
\begin{eqnarray*}
    |A\Delta P|&\leq& |A^*\Delta P|+|A\Delta A^*| \leq \frac{\gamma}{2^{11}}\min(|A^*|,|B^*|)+\mu|A|\\
    &\leq& \frac{\gamma}{2^{11}}\min(|A|,|B|)+\mu|A|
    \leq \frac{\gamma}{2^{10}}\min(|A|,|B|).
\end{eqnarray*}
and
\begin{eqnarray*}
    |B\Delta Q|&\leq& |B^*\Delta Q|+|B\Delta B^*|
    \leq \frac{\gamma}{2^{11}}\min(|A^*|,|B^*|)+\mu|B|\\
       &\leq& \frac{\gamma}{2^{11}}\min(|A|,|B|)+\mu|B|
    \leq \frac{\gamma}{2^{10}}\min(|A|,|B|).
\end{eqnarray*}

This completes the proof of Theorem~\ref{all_thm2}.
\end{proof}

\section{Proof of Theorem 1}

We start with a sketch of how the proof of Theorem~\ref{all_thm1} will proceed.

\vspace{5pt}
By Theorem~\ref{all_thm2}, we may assume that $A$ and $B$ are close to
intervals $I$ and $J$. Some delicate analysis around the endpoints of $I$ and
$J$, where we may slightly alter these intervals, will allow us to reduce to
the case where $A$ and $B$ are actually contained in $I$ and $J$. Thus there is
no `wraparound', and so we may as well be working in ${\mathbb Z}$ instead of ${\mathbb Z}_p$.

Assume for simplicity that $I = J$ has size $n$, that $A$ and $B$ are
contained in $I$ and $J$ and have size at least say $(1- 1/1000) n$, and that
$c_1 = c_2= c \sqrt{n}$ for some large constant $c$.
Fix $d$ to be about $\sqrt n$, and as usual write $X^y$ for $X \cap (y \mod d)$. Assume $A^0$ is the
largest of the fibres $A^y$. Then clearly $2|A^0| > |B^y|$ for all $y$. By
Theorem~\ref{CD_technical}, if $B'^y$ is a bounded set of $r$ random
points in $B^y$, chosen according to some distribution independent of $A$, we
have that $|A^0 + B'^y| \geq |A^0| + |B^y|-1$.

If we let $B'= \cup_y B'^y$, we have  $|A^0 + B'| \geq |A^0| |\pi(B)| +|B| -
| \pi(B)|$  and $A^0 + B' \subset \pi^{-1} \pi(B)$.

Now pick a random fibre $y \mod d$ and let $A' = A^0 \cup A^y$.
We want to show that
for each fibre $z \mod d$ we have ${\mathbb E} |(A'^y+B')^z| \geq n/d$.

Indeed, once we have shown this, then combining the two inequalities and summing
over all $z$ not in  $\pi(B)$ gives
${\mathbb E} |A'+B'| \geq |A^0| |\pi(B)| +|B| - |\pi(B)| + (d- |\pi(B)|) n/d \geq n/d
| \pi(B)| +|B| - |\pi(B)| + (d- |\pi(B)|) n/d \geq n+|B| - d$.
So fix sets $A'$ and $B'$ which satisfy this bound. To finish from
here we just note that by adding $d$ extra points in $A'$ and $d$ extra points in
$B'$ we can guarantee $|A'+B'| \geq |A|+|B|-1$. We also note $|A'| = |B'| =
O(\sqrt n)$.

To show ${\mathbb E} |(A'^y+B')^z| \geq n/d$, we proceed as follows. Note that the
proportion of fibres $y \mod d$ such that $A^y$ and $B^y$ have size at least
$9n/10d$ is at least $9/10$.  Therefore, for a fixed fibre $z \mod d$, with
probability at least $8/10$ both $A^y$ and $B^{z-y}$ have size at least $9n/10d$.
Conditioned on this event, it follows that $2|A^y| \geq |B^{z-y}|$. Using
Theorem~\ref{CD_technical} again, we obtain $|A^y+B'^{z-y}|
\geq |A^y|+|B^{z-y}|-1 \geq 18n/10d - 1$.
Hence
\[
{\mathbb E} |(A'^y+B')^z| \geq (8/10) (18n/10d - 1) \geq n/d.
\]

\vspace{5pt}
We now start to work towards the proof of Theorem~\ref{all_thm1}. We collect together in advance
some results that we shall need. The first of these results will be applied when we already know that our
sets are close to intervals.

\begin{theorem}\label{all_thm3}
There exists $\gamma>0$ such that for all $\alpha>0$ there exists $c$ for which the following holds.  Let $X$ and $Y$ be
subsets of two intervals $I$ and $J$ of $\mathbb{Z}_p$, and let $1 \leq c_1 \leq |X|$ and $1 \leq c_2 \leq |Y|$ be
integers such that $c_1c_2 \geq c\max(|X|,|Y|)$. Suppose that $\alpha |J| \leq |I| \leq \alpha^{-1}|J|$, \  $|I| +|J| \leq p$
and $\max(|I\setminus X|, |J \setminus Y| ) \leq \gamma \min(|I|, |J|)$. Then there exist
$X' \in X^{(c_1)}$ and $Y' \in Y^{(c_2)}$ such that $|X'+Y'| \geq |X|+|Y|-1$.
\end{theorem}

\begin{proof}
Since $|I| +|J| \leq p$, we may assume that the ambient space is $\mathbb{Z}$ rather than $\mathbb{Z}_p$.
Let $\gamma'$ and $k$ be the outputs of Theorem~\ref{CD_technical} with input $\alpha/2$ (since we are now in
$\mathbb{Z}$, there is no $\beta$, or more formally we are applying Theorem~\ref{CD_technical} inside $\mathbb{Z}_q$ for
some much larger $q$ with say $\beta=1/2$).
By increasing $\gamma'$ if
necessary we may assume that
$k \geq 100 /\gamma'$. Set $\gamma=\gamma'/100$ and
let $t=\lceil\log_{2/3}(1-(1+\alpha \gamma'/100)^{-1}) \rceil$, and put $c=2^5t(k+1)$.  We may
assume by symmetry  that
$c_1 \geq c_2$.   Let $d= \lfloor c_2(k+1)^{-1}\rfloor$. Note that the hypothesis forces $c_2 \geq c \geq 2(k+1)$, which
ensures that $d$ is positive. \\

The definition of $d$ and the inequality $k \ge 100/\gamma'$ imply
\begin{equation}\label{03.005}
    \min(|I|,|J|) \geq \min(|X|,|Y|) \geq c_2 \geq (k+1)d \geq 100d/ \gamma'.
\end{equation}

Given a set $Z$, recall that we write $Z^x_d$ for $Z \cap (x \text{ mod } d)$, the points of $Z$ in a fibre. Since $I$ and $J$ are intervals,  for every $x \in \mathbb{Z}_d$ we have
\begin{equation*}\label{03.025}
     |I|/d+1 \geq  |I^x_d| \geq |I|/d-1   \text{ and } |J|/d+1 \geq |J^x_d| \geq |J|/d-1.
\end{equation*}
Combining the last two inequalities, for every $x \in \mathbb{Z}_d$ we have
\begin{equation}\label{03.027}
     |I^x_d|  \geq |I|/d- (\gamma'/100 d) \min(|I|,|J|)  \text{ and }  |J^x_d| \geq |J|/d-(\gamma'/100 d) \min(|I|,|J|)
\end{equation}
and
\begin{equation}\label{03.028}
     |I^x_d|  \leq |I|/d+ (\gamma'/100 d) \min(|I|,|J|)  \text{ and }  |J^x_d| \leq |J|/d+(\gamma'/100 d) \min(|I|
,|J|).
\end{equation}

\noindent
We may assume (by taking a translate of $X$, if necessary) that
$|X^0_{d}| = \max_{x \in \mathbb{Z}_{d}}|X^x_{d}|$. Then
\begin{equation*}\label{03.01}
    |X^0_{d}| \geq |X|/d \geq |I|/d- (\gamma/ d) \min(|I|,|J|) \geq  |I|/d-  (\gamma'/3 d) \min(|I|,|J|).
\end{equation*}

Now define the sets $E_X, E_Y \subset \mathbb{Z}_d$ by
\[
E_X=\{x \in \mathbb{Z}_d \text{ : } |X^x_{d}| \geq (|I|/d)- (\gamma'/3 d) \min(|I|,|J|) \}
\]
and
\[
E_Y= \{x \in \mathbb{Z}_d \text{ : } |Y^x_{d}| \geq (|J|/d)- (\gamma'/3 d) \min(|I|,|J|) \}.
\]

For all $x\in E_X$ and $y \in E_Y$, noting that $X^x_d \subset I^x_d$ and $Y^y_d \subset J^y_d$, we see by  \eqref{03.027} and
\eqref{03.028} that
\[
\max(|X^x_d \Delta I^x_d|, |Y^x_d \Delta J^y_d|) \leq (1/3+1/100) (\gamma'/d) \min(|I|,|J|) \leq \gamma'
\min(|I^x_d|, |J^y_d|).
\]
Since $(\alpha/2)|J^y_d| \leq |I^x_d| \leq  (2/\alpha)|J^y_d|$, by Theorem~\ref{CD_technical}, there is a family $\mathcal{F}^y_d$ of subsets of $Y^y_d$ of size $k$, depending only on $Y^y_d$ (not on $X^x_d$), such that
\begin{equation}\label{03_06}
\E_{Z\in \mathcal{F}^y_d} |X^x_d+Z| \geq |X^x_d|+|Y^y_d|-1.
\end{equation}

Now construct a family $ \mathcal{F} = \{ \cup_{y \in E_Y} F^y_d \text{ : } F^y_d \in \mathcal{F}^y_d \}$,
and note that each set $F \in \mathcal{F}$ satisfies $|F| \leq |E_Y| k \leq dk \leq c_2-d$.
Define sets $ E_x' \subset E_X$ and $E_Y' \subset E_Y$ by
\[
E_X'=\{x \in \mathbb{Z}_d \text{ : } |X^x_{d}| \geq (|I|/d)- (\gamma'/10 d) \min(|I|,|J|) \}
\]
and
\[
E_Y'= \{x \in \mathbb{Z}_d \text{ : } |Y^x_{d}| \geq (|J|/d)- (\gamma'/10 d) \min(|I|,|J|) \}.
\]
By  Markov's inequality,
\begin{equation*}
  \mathbb{P}(E_X') \geq 1-\frac{\gamma}{(1/10-1/100)\gamma'}
  \geq \frac{2}{3} \ \  \text{and} \ \ \mathbb{P}(E_Y') \geq 1- \frac{\gamma}{(
  1/10-1/100)\gamma'} \geq \frac{2}{3}.
\end{equation*}
Simple calculations using \eqref{03.027} and \eqref{03.028} now show that for all $x' \in
E_X'\subset E_X, y' \in E_Y'\subset E_Y, x \in \mathbb{Z}_d$ and $y \in \mathbb{Z}_d \setminus E_y  $ we have
\begin{eqnarray*}
       |X^{x'}_d|+|Y^{y'}_d| &\geq& |I|/d +|J|/d -(\gamma'/5 d)  \min(|I|,|J|)\\
       &\geq& [|I|/d + (\gamma'/100 d) \min(|I|,|J|)] + [|J|/d-  (\gamma'/3d) \min(|I|,|J|)]\\
       &+& (1/3-1/5-1/100 )(\gamma'/d) \min(|I|,|J|)\\
       &\geq& |X^x_d|+|Y^y_d| +(\gamma'/10 d) \min(|I|,|J|)\\
       &\geq& |X^x_d|+|Y^y_d| +(\alpha \gamma'/10 d) \max(|I|,|J|)
       \geq (1+\alpha \gamma'/100) (|X^x_d|+|Y^y_d|).
\end{eqnarray*}

\noindent
Now consider the family
$ \mathcal{G} = \{ X^0_d\cup_{i=1}^t X^{x_i}_d \text{ : } x_i \in E_X'\}$.
By \eqref{03.028}, every set $G \in \mathcal{G}$ satisfies
\begin{eqnarray*}
        |G| &\leq& (t+1)(|I|/d+ (\gamma'/100 d) \min(|I|,|J|)) \leq 8td^{-1}|X|\\
        &\leq& 2^{4}t(k+1)c_2^{-1}|X| \leq 2^{-1}cc_2^{-1}|X| \leq  2^{-1}c_1\leq c_1-d.
\end{eqnarray*}

The last ingredient needed to complete the proof of Theorem~\ref{all_thm3} is the following lemma.

\begin{lemma}\label{F-and-G}
The families ${\mathcal F}$ and ${\mathcal G}$ are such that
$$\E_{X' \in \mathcal{G},\ Y'\in \mathcal{F} } |X'+Y'| \geq |X|+|Y|-d.$$
\end{lemma}
\begin{proof}
By the linearity of expectation, it is enough to show that for all $z \in \mathbb{Z}_d$ we have
$$\E_{X' \in \mathcal{G}, \ Y'\in \mathcal{F} } |(X'+Y')^z_d| \geq |X^z_d|+|Y^z_d|-1.$$
First assume that $z \in E_Y$. Then, using \eqref{03_06}, we get
\begin{equation*}
\E_{X' \in \mathcal{G},\ Y'\in \mathcal{F} } |(X'+Y')^z_d| \geq \E_{Z\in \mathcal{F}^z_d } |(X^0_d+Z)^z_d|
  = \E_{Z\in \mathcal{F}^z_d } |X^0_d+Z|
    \geq |X^0_d|+|Y^z_d|-1 \geq |X^z_d|+|Y^z_d|-1.
\end{equation*}
Now assume instead that $z \not \in E_Y$. Then, using \eqref{03_06}, ${\mathbb P}(E_X')\ge 2/3$, and our bound on $|X^{x'}_d|+|Y^{y'}_d|$, we obtain
\begin{eqnarray*}
&&\E_{X' \in \mathcal{G}, \ Y'\in \mathcal{F} } |(X'+Y')^z_d| \geq
      \E_{\substack{x_1, \hdots, x_t \in E_X' \\ Y'\in \mathcal{F} }} |((\cup_iX^{x_i}_d)+Y' )^z_d|\\
&& \hspace{30pt}\geq \E_{\substack{x_1, \hdots, x_t \in E_X' \\ Y'\in \mathcal{F} }}
      \bigg( |((\cup_iX^{x_i}_d)+Y' )^z_d| \text{ : } \exists i \text{ such that } z-x_i \in E_Y' \bigg)  \\
       && \hspace{130pt}  \times\Prob_{x_1, \hdots, x_t \in E_X'} \bigg(\exists i \text{ such that } z-x_i \in E_Y'  \bigg)\\
         &&\hspace{30pt}\geq \E_{\substack{x \in E_X' \cap (z-E_Y') \\ Y'\in \mathcal{F} }}  |(X^{x}_d+Y' )^z_d|    (1-(2/3)^t)
        \geq \E_{\substack{x \in E_X' \cap (z-E_Y') \\ Z\in \mathcal{F}^{z-x}_d }}  |X^{x}_d+Z |   (1-(2/3)^t)\\
        &&\hspace{30pt} \geq (|X^x_d|+|Y^{z-d}_d|-1) (1-(2/3)^t)
           \geq \bigg((1+\alpha \gamma'/100) (|X^0_d|+|Y^z_d|)-1\bigg)(1-(2/3)^t)\\
        &&\hspace{30pt}\geq (1+100^{-1}\alpha \gamma') (|X^0_d|+|Y^z_d|)(1-(2/3)^t)-1
          \geq |X^0_d|+|Y^z_d|-1 \geq |X^z_d|+|Y^z_d|-1,
\end{eqnarray*}
proving Lemma~\ref{F-and-G}.
\end{proof}

To complete the proof of Theorem~\ref{all_thm3}, note that if
$X' \in X^{(c_1-d)}$ and $Y' \in Y^{(c_2-d)}$ satisfy
$|X'+Y'| \geq |X|+|Y|-d$, then
there exist sets $X'' \in X^{(c_1)}$ and $Y'' \in Y^{(c_2)}$ such that
$|X''+Y''| \geq |X|+|Y|-1.$
\end{proof}

Having proved Theorem~\ref{all_thm3}, we turn to `improving the setup by changing the ends of the
intervals', the step we mentioned in our sketch of the proof at the start of the section.

\begin{lemma}\label{all_lem4}
  For all $\beta$ and $\gamma$ with $2^{-20}>  \beta >2^{20} \gamma >0$ the following holds. Let $A$ and $B$ be
  subsets of $\mathbb{Z}_p$ and
  let $I$ and $J$ be intervals of $\mathbb{Z}_p$ such that
$|A|+|B| \leq (1-\beta)p$ and $\max(|A\Delta I|, |B\Delta J|) \leq \gamma \min(|A|, |B|)$.
Then in $\mathbb{Z}_p$ there are two sets of three consecutive intervals, $I_1, I_2, I_3$ and $J_1, J_2, J_3$, with
\[
\lfloor  (\beta/8) p \rfloor \le    \min(|I_1|, |I_3|,|J_1|,|J_3|)\le \max(|I_1|, |I_3|,|J_1|,|J_3|) \leq  (\beta/4) p
\]
such that, setting $A_i=A\cap I_i$ and $B_i=B \cap J_i$, we have $|A_1|=|B_1|$, $|A_3|=|B_3|$, and
\begin{equation}\label{all_eq7}
\max(|A\Delta I_2|, |B\Delta J_2|) \leq 2^{10}\gamma \min(|A|, |B|).
\end{equation}
Moreover, for $i \in \{1,3\}$ and for any two arithmetic progressions $P$ and $Q$ we have
\begin{equation*}
    \max(|A_i \Delta P|, |B_i \Delta Q |) \geq  |A_i|/2^{10}=  |B_i|/2^{10}.
\end{equation*}
\end{lemma}
\begin{proof}
  Let $I_2'$ and $J_2'$ be maximal intervals in $\mathbb{Z}_p$  satisfying \eqref{all_eq7}: note that such intervals do exist by
  hypothesis. By our bounds on $|A|+|B|$, $|A\Delta I|$ and $|B\Delta J|$,  we have
\begin{equation*}\label{eqp_001}
    |I_2'|+|J_2'| \leq |A|+|B|+ 2 \gamma \min(|A|,|B|) \leq (1-\beta +2\gamma) p \leq (1- (\beta/2)) p.
\end{equation*}
Thus we can construct intervals $I_1', I_3'$ and $J_1', J_3'$ such
that  $I_1',I_2',I_3'$ and $J_1', J_2', J_3'$ are two families of three consecutive
 intervals of $\mathbb{Z}_p$ with
\begin{equation}\label{eqp_002}
|I_1'|=|I_3'|=|J_1'|=|J_3'|= \lfloor (\beta/8) p \rfloor.
\end{equation}
Note that, by  the maximality of $I_2'$ and $J_2'$, for any intervals $I_1''', I_3'''$ and $J_1''', J_3'''$
such that $I_1''', I_2', I_3'''$ and $J_1''', J_2', J_3'''$ are familes of consecutive
intervals, if we set $A_i'''=A \cap I_i'''$ and $B_i'''= B \cap J_i'''$, then we have
\begin{equation}\label{eqp_003}
    |I_i'''| > 2 |A_i'''| \text{ and } |J_i'''| > 2 |B_i'''|.
\end{equation}

Let $A_i'=A \cap I_i'$ and $B_i'=B\cap J_i'$, and assume that
$|A_1'| \geq |B_1'| \text{ and } |A_3'| \geq |B_3'|$.
(The three other cases are analogous.)
Note that by \eqref{all_eq7} we have
\begin{equation}\label{eqp_0030}
 \max(|A_1'|, |A_3'|, |I_2' \setminus A_2'|)  \leq  \gamma \min(|A|, |B|)  \leq 2\gamma \min( |A_2'|, |B_2'|) .
\end{equation}
Now consider subintervals $I_1''$ and $I_3''$ at the ends of interval $I_2'$ that are minimal subject
to the following two properties:
setting $A_i''=A \cap I_i''$, we have
\begin{equation}\label{eqp_004}
|A_1''| \geq 14|A_1'| \text{ and } |A_3''|\geq 14|A_3'|
\end{equation}
and there exist intervals $P_1''$ and $P_3''$ (contained inside $I_1''$ and $I_3''$ respectively) such that
\begin{equation}\label{eqp_005}
|P_1'' \Delta A_1''| \leq |A_1''|/14 \text{ and } |P_3'' \Delta A_3''|\leq |A_3''|/14.
\end{equation}

(Here we insist that if $A_1'=\emptyset$ or $A_3'=\emptyset$ then $I_1''=P_1''=A_1''=
\emptyset$ or $I_3''=P_3''=A_3''=\emptyset$, respectively.)
 Note that such intervals $I_1''$ and $I_3''$ do exist, as the intervals $I_1''=P_1''=I_3''=P_3''=I_2'$ have the desired properties by \eqref{eqp_0030}.

 We now show that
\begin{equation}\label{eqp_006}
|I_1''|+|I_3''| \leq 2^6\gamma \min(|A_2'|, |B_2'|)<|I_2'|.
\end{equation}

The right-hand inequality is immediate, since $A_2' \subset I_2'$ and $\gamma < 2^{-6}$. For the left-hand
inequality, suppose for a contradiction that say $ |I_1''| > 2^5 \gamma \min(|A_2'|,|B_2'|)$. Consider first the case when $2 \gamma \min(|A_2'|,|B_2'|) <1
$. By \eqref{eqp_0030} we have $A_1'=A_3'=\emptyset$, and hence we obtain $I_1''=I_3''=\emptyset$, which is a contradiction.

Consider now the case when $2 \gamma \min(|A_2'|,|B_2'|)   \geq 1$. In this case we consider the proper subinterval $I_1'''$ at the
end of $I_1''$ with $|I_1'''|=\lfloor 2^5\gamma  \min(|A_2'|,|B_2'|)  \rfloor
\geq 30  \gamma  \min(|A_2'|,|B_2'|)$. Let $A_1'''=I_1''' \cap A$. By \eqref{eqp_0030}, we
have $| I_1''' \setminus A_1'''| \leq |I_2' \setminus A_2'| \leq 2 \gamma \min(|A_2'|,|B_2'|)$. In particular,
we have $|A_1'''| \geq 28 \gamma \min(|A_2'|, |B_2'|) $, which implies $|I_1''' \setminus A_1'''| \leq |A_1'''|/14$. But by \eqref{eqp_0030},
 we also have $|A_1'| \leq 2\gamma \min(|A_2'|, |B_2'|)$, which implies $|A_1'''| \geq 14|A_1'|$.

Therefore $I_1'''$ is an interval strictly smaller than $I_1''$ with the desired properties, giving a contradiction. This proves
 inequality \eqref{eqp_006}.

 By \eqref{eqp_006}, the intervals $I_1''$ and $I_3''$ induce a partition $I_2'=I_1''\sqcup I_2 \sqcup I_3''$ into consecutive
 intervals. Moreover, by \eqref{eqp_006} and \eqref{all_eq7} we get
\begin{equation}\label{eqp_007}
|A\Delta I_2| \leq |A\Delta I_2'| + |I_1''|+|I_3''| \leq 2^7 \gamma \min(|A_2'|,|B_2'|).
\end{equation}

Note also that by \eqref{all_eq7} we have
\begin{equation}\label{eqp_0031}
|J_2'\setminus B_2'| \leq \gamma \min(|A|,|B|) \leq 2 \gamma \min(|A_2'|,|B_2'|).
\end{equation}

Now consider subintervals $J_1''$ and $J_3''$ at the ends of $J_2'$ such that with $B_i''=B\cap J_i''$ we
have
\begin{equation}\label{eqp_0032}
|B_1'|+|B_1''|=|A_1'|+|A_1''| \text{ and } |B_3'|+|B_3''|=|A_3'|+|A_3''|
\end{equation}
Note that there are such intervals, since for $i \in \{1,3\}$ both $|B_i'| \leq |A_i'| \leq |A_i'|+|A_i''|$
and $|B_i'|+|B_2'| \geq |B_2'| \geq |A_i'|+|I_i''| \geq |A_i'|+|A_i''|$ hold. We now show that
\begin{equation}\label{eqp_0040}
|J_1''|+|J_3''| \leq 2^8\gamma \min(|A_2'|, |B_2'|)<|J_2'|.
\end{equation}

Assume for a contradiction that $|J_1''| \geq 2^7 \gamma \min(|A_2'|, |B_2'|)$. On the one hand, by \eqref{eqp_0031}
we deduce $|B_1''| \geq |J_1''|- |J_2' \setminus B_2'| \geq 126 \gamma \min(|A_2'|,|B_2'|)$. On the other hand,
by   \eqref{eqp_0030} we have $|A_1'| \leq 2 \gamma
\min(|A_2'|, |B_2'|)$ and by \eqref{eqp_006} we have $|A_1''| \leq |I_1''| \leq 2^6 \gamma \min(|A_2'|, |B_2'|)$.
Thus we
obtain $|B_1''| > |A_1'|+|A_1''|$, which gives the desired contradiction. This proves inequality \eqref{eqp_0040}.

By \eqref{eqp_0040}, the intervals $J_1''$ and $J_3''$ induce a partition $J_2'=J_1''\sqcup J_2 \sqcup J_3''$ into
consecutive intervals. Moreover, by \eqref{eqp_0040} and \eqref{all_eq7} we have
\begin{equation}\label{eqp_008}
|B\Delta J_2| \leq |B\Delta J_2'| + |J_1''|+|J_3''| \leq 2^9 \gamma \min(|A_2'|,|B_2'|).
\end{equation}

For $i \in \{1,3\} $, set $I_i=I_i' \sqcup I_i''$, $J_i= I_i'\sqcup I_i''$, and note that $I_1, I_2, I_3$ and $J_1, J_2, J_3$ are consecutive intervals. Moreover, by \eqref{eqp_002} we have
\begin{equation}\label{eqp_009}
\min(|I_1|, |I_3|, |J_1|, |J_3|) \geq \lfloor (\beta/8) p \rfloor.
\end{equation}
In addition, by \eqref{eqp_002}, \eqref{eqp_006} and \eqref{eqp_0031}, we also have
\begin{equation}\label{eqp_0010}
\max(|I_1|, |I_3|, |J_1|, |J_3|) \leq  (\beta/4) p .
\end{equation}
For $i \in \{1,2,3\}$ let $A_i=A\cap I_i$ and $B_i = B \cap J_i$. By \eqref{eqp_0032} we have
\begin{equation}\label{eqp_010}
|A_1|=|B_1| \text{ and } |A_3|=|B_3|.
\end{equation}

It remains to show that for any arithmetic progressions $P_1$ and $P_3$ we have
\begin{equation}\label{eqp_011}
|A_1 \Delta P_1| \geq 2^{-10}|A_1| \text{ and } |A_3 \Delta P_3| \geq 2^{-10}|A_3|.
\end{equation}
Assume for a contradiction that
\begin{equation}\label{eqp_012}
|A_1 \Delta P_1| < 2^{-10}|A_1|,
\end{equation}
which in particular means that
\begin{equation}\label{eqp_0125}
A_1\neq \emptyset \text{ i.e. } A_1' \neq \emptyset.
\end{equation}
Note that \eqref{eqp_004}, \eqref{eqp_005} and \eqref{eqp_012} imply
\begin{eqnarray*}
    |P_1'' \Delta P_1| &&\leq |A_1 \Delta P_1| + |P_1'' \Delta  A_1|
     \leq |A_1 \Delta P| + |A_1'| + |P_1'' \Delta A_1''| \\
    &&\leq |A_1|/2^{10} + |A_1''|/14+|A_1''|/14 \leq (1/2^9+1/7) |A_1''| \le |P_1''|/4.
\end{eqnarray*}

Recall that, when $A_1' \neq \emptyset$, $P_1''$ is a subinterval of $I_1''$ with $p/8 \geq |P_1''| \geq 8$
by \eqref{eqp_004}, \eqref{eqp_005} and \eqref{eqp_006}. It follows that $P_1$ is an interval intersecting the
interval $I_1''$ of size $p/4 \geq |P_1| \geq 4$. By \eqref{eqp_0010} we also
deduce $|P_1|+|I_1| \leq p/2$, which implies
that $P_1 \cap I_1$ is an interval. By replacing $P_1$ with $P_1\cap I_1$, we may assume that $P_1$ is a
subinterval of $I_1$.

We distinguish two cases: recalling that $I_1''$ and $I_3''$ were chosen to be minimal subject
to \eqref{eqp_004} and \eqref{eqp_005}, we ask whether the lower bound on the size of $A_1''$ in \eqref{eqp_004} is attained or not.

\vspace{5pt}
\textbf{Case A. $|A_1''|=14|A_1'|$.}\\

In this case, by \eqref{eqp_003}, we have
\begin{equation*}\label{eqp_015}
|(P_1 \cap I_1')\Delta A_1'| \geq |A_1'|.
\end{equation*}
So we obtain
\begin{equation}\label{eqp_016}
  |A_1\Delta P_1| \geq |(P_1 \cap I_1')\Delta A_1'| \geq |A_1'| = |A_1|/15.
\end{equation}

\vspace{5pt}
\textbf{Case B. $|A_1''|>14|A_1'|\geq 14$.}\\

In this case, by the minimality of $I_1''$, if we let $x$  be the last point inside $A_1''$, then by \eqref{eqp_005} we deduce
\begin{equation*}\label{eqp_018}
 |(P_1 \cap I_1'') \Delta  (A_1'' \setminus \{x\})| > |A_1'' \setminus \{x\}| /14 .
\end{equation*}
This is equivalent to
\begin{equation*}\label{eqp_019}
 |(P_1 \cap I_1'') \Delta  (A_1'' \setminus \{x\})| \geq \min(2, |A_1'' \setminus \{x\}| /14) ,
\end{equation*}
and hence
\begin{equation*}\label{eqp_020}
 |(P_1 \cap I_1'') \Delta  A_1''| \geq \min(1, (|A_1''|-1)/14-1) \geq |A_1''|/28.
\end{equation*}
Thus we obtain
\begin{equation}\label{eqp_021}
 |A_1\Delta P_1| \geq |(P_1 \cap I_1'')\Delta A_1''| \geq |A_1''|/28 \geq |A_1|/56.
\end{equation}

Inequalities \eqref{eqp_016} and \eqref{eqp_021} imply that, in either case,
$|A_1 \Delta P_1| \geq |A_1|/56$ , contradicting \eqref{eqp_012} and so proving  \eqref{eqp_011}.
The proof of Lemma~\ref{all_lem4} is now complete, thanks to \eqref{eqp_009}, \eqref{eqp_0010} \eqref{eqp_010},
\eqref{eqp_007}, \eqref{eqp_008} and \eqref{eqp_011}.
\end{proof}

Our next lemma is a simple fact about the `stickout' of sumsets from a set of fixed size:
roughly speaking, it says that if $Y$ is much larger than $X$ then the sum of $X$ with a random few points of
$Y$ is expected to be
much larger than $2|X|$.

\begin{lemma}\label{all_lem7}
  Let $X$, $Y$ and $Z$ be subsets of $\mathbb{Z}_p$ and let $1 \leq c_1 \leq |X|$ and $1 \leq c_2 \leq |Y|$ be
  integers such that $c_1c_2 \geq 16 |X| $. Suppose that $8|X|, 8|Z| \leq |Y|<p/2$.
Then there exist
$X' \in X^{(c_1)}$ and $Y' \in Y^{(c_2)}$ such that $|(X'+Y')\setminus Z| \geq 2|X|$.
\end{lemma}

\begin{proof}
Suppose the assertion is false, i.e. $\max_{X' \in X^{(c_1)}, Y' \in Y^{(c_2)}} |(X'+Y')\setminus Z| < 2|X|$ which, in particular, implies that $\max_{X' \in X^{(c_1)}, Y' \in Y^{(c_2/2)}} |(X'+Y')\setminus Z| < 2|X|$. Let $X'=\{x_1, \hdots, x_{c_1}\}$ and $Y'=\{y_1, \hdots, y_{c_2/2}\}$ be elements of $X^{(c_1)}$ and $Y^{(c_2/2)}$ chosen uniformly at random. Then  ${\mathbb E} |(X'+Y')\setminus Z|$ is bounded from below as follows:
\begin{eqnarray*}
    && \ \ \   \sum_{i,j} {\mathbb E}|\{x_i+y_j\}\, \setminus \, [Z \cup (\{x_1, \hdots, x_{i-1}, x_{i+1}, \hdots, x_{c_1}\} + \{y_1, \hdots, y_{j-1}, y_{j+1}, \hdots, y_{c_2}\} ) ]  |\\
    &&= \sum_{i,j} \Prob | x_i+y_j \not \in  Z \cup (\{x_1, \hdots, x_{i-1}, x_{i+1}, \hdots, x_{c_1}\} + \{y_1, \hdots, y_{j-1}, y_{j+1}, \hdots, y_{c_2}\} ) |\\
    &&\geq \sum_{i,j} 1-\max_{ X' \in X^{(c_1)}, Y'\in Y^{(c_2/2)}}\frac{|Z \cup (X'+Y')|}{|Y|-(c_2/2)}\\
    &&\geq \sum_{i,j} 1- \frac{|Z|+2|X|}{|Y|/2} \geq \frac{c_1c_2}{2}  (1-\frac{3}{4}) \geq 2|X|,
\end{eqnarray*}
so $ {\mathbb E} |(X'+Y')\setminus Z| \geq 2|X|$, completing the proof.
\end{proof}

As an immediate corollary we have the following, obtained by applying the previous lemma inductively on $k$
(increasing $Z$ at each stage).

\begin{corollary}\label{cor_new000}
  Let $X_1, \hdots, X_k$, $Y_1, \hdots, Y_k$ and $Z$ be subsets of $\mathbb{Z}_p$ and
  let $1 \leq c_1^i \leq |X_i|$ and $1 \leq c_2^i \leq |Y_i|$ be integers such that
  for all $i$ we  have $c_1^ic_2^i\geq 16|X_i|$. Suppose
  that for all $i$ we have $16|X_i|, 16|Z|\leq |Y_i| <p/2$. Then there exist
  $X'_i \in (X_i)^{(c_1^i)}$ and $Y'_i \in (Y_i)^{(c_2^i)}$ (for each $i$) such that
\[
  \ \ \ \ \ \ |\cup_i(X'_i+Y'_i)\setminus Z|
  \geq \min(16^{-1}|Y_1|, \hdots, 16^{-1}|Y_k|, 2\sum_i|X_i|). \ \ \ \ \ \ \square
\]
\end{corollary}

\noindent
The final ingredient we need is a somewhat cumbersome result about partitions into intervals.

\begin{lemma}\label{all_lem6}
For all $1/2^{10}> \alpha, \beta >0$ the following holds. Consider partitions $\mathbb{Z}_p= I_0 \sqcup I_1
\sqcup I_2 \sqcup I_3 =J_0 \sqcup J_1 \sqcup J_2 \sqcup J_3$ into consecutive intervals such that
$|I_2|+|J_2| \leq (1-\beta/2)p$, $(\alpha/2)|J_2| \leq |I_2| \leq (2/\alpha)  |J_2|$,
$\min(|I_2|, |J_2|) \ge 24/\beta$ and
\[
\lfloor (\beta/8) p \rfloor \le \min(|I_1|,|I_3|,|J_1|,|J_3|) \le \max(|I_1|, |I_3|,|J_1|,|J_3|) \le (\beta/4) p.
\]
Then there are four families  of subsets of $\mathbb{Z}_p$, each of size $k \leq 100/ \alpha \beta$,
\[
\mathcal{I}_0=\{I^1_0, \hdots, I^k_0\}, \ \ \mathcal{I}_2=\{I^1_2, \hdots, I^k_2\}, \ \
\mathcal{J}_0=\{J_0^1, \hdots,J_0^k\}, \ \ \mathcal{J}_2=\{J_2^1, \hdots, J_2^k\}
\]
such that $\cup_i I_0^i = I_0, \cup_i J_0^i = J_0\ \text{ and }\ \cup_i I_2^i \subset I_2, \cup_i J_2^i \subset J_2$.
Furthermore, for every $1 \leq i \leq k$ we have \ $|I_2^i|=|J_2^i|= \lfloor (\beta/24) \min(|I_2|, |J_2|) \rfloor$
and $(I_0^i+J_2^i) \cap (I_2+J_2) = (J_0^i + I_2^i) \cap (I_2+J_2) = \emptyset$.
\end{lemma}

\begin{proof}
  We construct the families of intervals $\mathcal{I}_0$ and $\mathcal{J}_2$. The construction of the
   $\mathcal{J}_0$ and $\mathcal{I}_2$ is identical.
By the conditions above, for every $x \in I_0$ there is a
  subinterval $J_2'$ of $J_2$ of size $\lfloor  (\beta/8) \min(|I_2|,|J_2|) \rfloor$ such that
$(x+J_2') \cap (I_2+J_2) =\emptyset.$

Let $\mathcal{J}_2$ be a maximal collection of disjoint intervals of size $\lfloor  (\beta/24) \min(|I_2|,|J_2|) \rfloor$
contained inside interval $J_2$.
First note that
\begin{equation*}
  |\mathcal{J}_2| \leq |J_2|/\lfloor (\beta/24) \min(|I_2|,|J_2|) \rfloor \leq (48/ \beta)
  (2/ \alpha) = 100/ (\alpha \beta).
\end{equation*}

Secondly, note that for any  subinterval $J_2'$ of $J_2$ of size $\lfloor (\beta/8) \min(|I_2|,|J_2|) \rfloor$
there exists $J_2'' \in \mathcal{J}_2$
such that $J_2'' \subset J_2'$. Therefore, for any point $x \in I_0$ there exists $J_2^x \in \mathcal{J}_2$ such that
$$(x+J_2^x) \cap (I_2+J_2) =\emptyset.$$

Let $\mathcal{J}_2=\{J_2^1, \hdots, J_2^k\}$.  For each $1 \leq i \leq k$
set $I_0^i=\{x \in I_0 \text{ : } J_2^x=J_2^i\}$, and finally put
$\mathcal{I}_0=\{I_0^1, \hdots, I_0^k\}$.
These families $\mathcal{I}_0$ and $\mathcal{J}_2$ have the desired properties.
\end{proof}

We are now ready to prove Theorem~\ref{all_thm1}.

\begin{proof}[Proof of Theorem~\ref{all_thm1}]
  Fix $\alpha, \beta>0$ and assume $\alpha, \beta<2^{-30}$. Let $\gamma'$ and $c'$ be the output of
  Theorem~\ref{all_thm3} with input $\alpha/2$. Let $\gamma=2^{-30}\min(\gamma',\beta)$.
  Let $\epsilon>0$ and $c''>0$ be the
  output of Theorem~\ref{all_thm2} with input $\alpha, \beta, \gamma$. Choose
  $c=10^{20} \alpha ^{-2} \beta^{-2}\max(c',c'')$.

  Let $A$ and $B$ be
  subsets of $\mathbb{Z}_p$ and let $1 \leq c_1 \leq |A|$ and $1 \leq c_2 \leq |B|$ such that
  $c_1 c_2 \geq c \max(|A|,|B|)$. This forces $\min(c_1,c_2) \geq c$, which, in particular, implies $\min(|A|,|B|) \geq c$. Let $c_1'=c_1 \alpha \beta/ 10^5 $
  and $c_2'=c_2 \alpha \beta/ 10^5 $ and note that $c_1'c_2' \geq 10^{10}\max(c',c'') \max(|A|,|B|)$ and
  $\min(c_1',c_2') \geq 10^{10}\max(c',c'')$.  Observe that we are done unless
\[
    \max_{A' \in A^{(c_1')}, B' \in B^{(c_2')}}|A'+B'| < |A|+|B|-1,
\]
so we may assume that this holds. But then we may apply Theorem~\ref{all_thm2} to deduce that
there are arithmetic progressions $I$ and $J$ with the same common difference such that
\begin{equation}\label{main_2}
    \max(|A\Delta I|, |B \Delta J|) \leq \gamma \min(|I|,|J|).
\end{equation}
Furthermore, we may and shall assume that $I$ and $J$ are intervals.
But then Lemma~\ref{all_lem4} can be used to deduce that $\mathbb{Z}_p$ has disjoint partitions into
consecutive intervals, $\mathbb{Z}_p=I_0\sqcup I_1\sqcup I_2 \sqcup I_3= J_0\sqcup J_1\sqcup J_2 \sqcup J_3$,
that have the following properties. (Here and elsewhere, the notation $\sqcup$ indicates that we are
taking the union of disjoint sets.) On the one hand we have
\begin{equation}\label{main_3}
    \lfloor  (\beta/8) p \rfloor \le  \min(|I_1|, |I_3|,|J_1|,|J_3|) \le \max(|I_1|, |I_3|,|J_1|,|J_3|) \leq  (\beta/4) p
\end{equation}
On the other hand, writing $A_i=A\cap I_i$ and $B_i=B \cap J_i$, we have $|A_1|=|B_1|$,  $|A_3|=|B_3|$, and
\begin{equation}\label{main_6}
    \max(|A\Delta I_2|, |B\Delta J_2|) \leq 2^{10}\gamma \min(|A|, |B|).
\end{equation}
Moreover, if $i \in \{1,3\}$ and $P$ and $Q$ are arithmetic progressions, then
\begin{equation}\label{main_7}
    \max(|A_i \Delta P|, |B_i \Delta Q |) \geq 2^{-10} |A_i|=2^{-10}  |B_i|.
\end{equation}
In particular, if $i \in \{1,3\}$ we have $|A_i|=|B_i|=0$ or $|A_i|=|B_i| \geq 2$. Because $\min(c_1',c_2') \geq c''$ and $c_1'c_2' \geq c''|A_i||B_i|$, we either have $\min(c_1',c_2') \geq |A_i|=|B_i|$ or $\min(c_1',|A_i|) \min(c_2', |B_i|) \geq c''|A_i|=c''|B_i|$. By the
contrapositive of Theorem~\ref{all_thm2}, we deduce that
\begin{equation*}
    \E_{A' \in A_i^{(c_1')}, \ B' \in B_i^{(c_2')}} |A'+B'| > |A_i|+|B_i|-1.
\end{equation*}
We further deduce that there are sets $A_1' \in A_{1}^{(c_1')},\ B_1' \in B_{1}^{(c_2')}, \
A_3' \in A_{3}^{(c_3')}$\ and \ $B_3' \in B_{3}^{(c_2')}$\ such that
\begin{equation}\label{main_008}
|A_1'+B_1'|\geq |A_1|+|B_1|\ \ \text{and} \ \ |A_3'+B_3'|\geq |A_3|+|B_3|.
\end{equation}
Choose a subset $Z$ of $\cup_{i \in \{1,3\}}(A_i'+B_i') $
with $|A_1|+|B_1|+|A_3|+|B_3|$ elements, which, by \eqref{main_6}, satisfies
\begin{equation}\label{ccc_000}
      |Z|=|A_1|+|B_1|+|A_3|+|B_3|\le |A\Delta I_2|+ |B \Delta J_2| \leq 2^{11} \gamma \min(|A|,|B|).
\end{equation}

Now,   by \eqref{main_6} we also have
\begin{equation*}
    (1- 2^{10} \gamma)|A| \leq  |I_2| \leq  (1+2^{10}\gamma) |A| \ \ \ \text{and} \ \ \
    (1- 2^{10} \gamma)|B| \leq  |J_2| \leq  (1+2^{10}\gamma) |B|,
\end{equation*}
and so
\begin{equation}\label{main_9}
    |I_2|+|J_2| \leq (1-\beta)(1+2^{10}\gamma) p \leq (1-\beta/2)p
\end{equation}
and
\begin{equation}\label{main_10}
    (\alpha/2) |I_2| \leq |J_2| \leq (2/ \alpha)|I_2|.
\end{equation}
Furthermore,
we deduce that
\begin{equation}\label{main_11}
 \min(|I_2|,|J_2|) \geq 2^{-1}\min(|A|,|B|) \geq 2^{-1}c \ge 24/\beta.
\end{equation}

First, by construction, we have $A_2 \subset I_2 \text{ and } B_2 \subset J_2$.
Hence, recalling the relations \eqref{main_6}, \eqref{main_9} and \eqref{main_10}, and the fact that $\min(c_1', |A_2|) \min(c_2', |B_2|) \geq 4^{-1}c_1'c_2' \geq c'\max(|A_2'|,|B_2'|)$,  we may apply Theorem~\ref{all_thm3}
with parameters $\alpha/2$, $\gamma'\geq 2^{30} \gamma$ and $c'$   to obtain the following: there exist $A_2' \in A_2^{(c_1')}$
and $B_2' \in B_2^{(c_2')}$ such that
\begin{equation}\label{ddd_000}
|A_2'+B_2'| \geq |A_2|+|B_2|-1.
\end{equation}

Second, by \eqref{main_3},  \eqref{main_9}, \eqref{main_10} and \eqref{main_11}, we may apply
Lemma~\ref{all_lem6} to find four families  of subsets of $\mathbb{Z}_p$, each with
$k\le 100/ \alpha \beta$ sets,
\[
\mathcal{I}_0=\{I^1_0, \hdots, I^k_0\}, \mathcal{I}_2=\{I^1_2, \hdots, I^k_2\}, \mathcal{J}_0=\{J_0^1,
\hdots, J_0^k\}, \mathcal{J}_2=\{J_2^1, \hdots, J_2^k\},
\]
such that
\begin{equation}\label{main_14}
\cup_i I_0^i = I_0,\ \ \cup_i J_0^i = J_0 \ \ \ \text{ and }\ \ \ \cup_i I_2^i \subset I_2, \ \ \ \cup_i J_2^i \subset J_2
\end{equation}
and for every $1 \leq i \leq k$ and we have
\begin{equation}\label{main_15}
|I_2^i|=|J_2^i|= \lfloor  (\beta/24) \min(|I_2|, |J_2|) \rfloor
\end{equation}
and
\begin{equation}\label{main_16}
(I_0^i+J_2^i) \cap (I_2+J_2) = (J_0^i + I_2^i) \cap (I_2+J_2) = \emptyset.
\end{equation}
Writing $A_j^i= A \cap I_j^i$ and $B_j^i=B \cap J_j^i$, by \eqref{main_6}, we have
\begin{equation*}\label{ccc_002}
    \max \{|A_0^i|, |B_0^i|\} \leq 2^{10} \gamma \min(|A|,|B|) .
\end{equation*}
and by \eqref{main_6},   \eqref{main_11} and \eqref{main_15},  we have
\begin{equation}\label{ccc_001}
    \big(\beta/100\big)\ \min(|A|,|B|) \leq \min \{ |A_2^i|, |B_2^i|\} \le \max  \{ |A_2^i|, |B_2^i|\} \leq p/2.
\end{equation}
The last two  inequalities and inequality \eqref{ccc_000} imply that
\begin{equation}\label{ccc_0025}
    10^{10} \max \{|A_0^i|,\ |B_0^i|,\ |Z|\} \leq  \min \{ |A_2^i|, |B_2^i|\} \le \max  \{ |A_2^i|, |B_2^i|\} \leq p/2.
\end{equation}
As we have $\min(c_1',c_2') \geq 10^{10}$ and $c_1'c_2' \geq 10^{10} \max(|A|,|B|) \geq 10^{10}\max(|A_0^i|,|B_0^i|)$, we further  deduce that
\begin{equation}\label{ccc_003}
    \min(c_1',|A_0^i|) \min(c_2', |B_2^i|) \geq 16 |A_0^i| \ \ \text{and} \ \  \min(c_1',|B_0^i|) \min(c_2', |A_2^i|) \geq 16 |B_0^i|.
\end{equation}

By Corollary~\ref{cor_new000} together with \eqref{ccc_0025} and \eqref{ccc_003} we deduce that for $i \in [k]$ there exist $A'^i_0 \in (A_0^i)^{(c_1')}, B'^i_2 \in (B_2^i)^{(c_2')}, A'^i_2 \in (A_2^i)^{(c_1')}, B'^i_0 \in (B_0^i)^{(c_2')}$ such that
\begin{equation*}
     |\cup_i(A'^i_0+B'^i_2)\cup_i(A'^i_2+B'^i_0)\setminus Z| \geq 16^{-1}\min(|A_2^1|, \hdots, |A_2^k|,|B_2^1|, \hdots, |B_2^k|, 32\sum_i|A^i_0|+|B^i_0|)
\end{equation*}
By \eqref{main_6} and \eqref{ccc_001}, we further deduce
\begin{equation}\label{ccc_005}
     |\cup_i(A'^i_0+B'^i_2)\cup_i(A'^i_2+B'^i_0)\setminus Z| \geq |A_0|+|B_0|
\end{equation}

Finally, note that from \eqref{main_3} and \eqref{main_9} it follows that
\[
I_1+J_1, I_2+J_2 \text{ and } I_3+J_3
\]
are disjoint sets, which in particular implies that
$$A_1'+B_1', A_2'+B_2' \text{ and } A_3'+B_3'$$
are disjoint sets. Moreover, by \eqref{main_16}, it follows that
$$\cup_i(A'^i_0+B'^i_2)\cup_i (A'^i_2+B'^i_0) \text{ and }  A_2'+B_2' $$
are disjoint sets.
Let $$A'= A_1'\cup A_2'\cup A_3' \cup_i A'^i_0 \cup_i A'^i_2 \text{ and } B'= B_1'\cup B_2'\cup B_3' \cup_i B'^i_0 \cup_i B'^i_2.$$
From \eqref{main_008}, \eqref{ccc_000}, \eqref{ddd_000} and \eqref{ccc_005} we conclude that
$$|A'+B'| \geq (|A_0|+|B_0|)+(|A_1|+|B_1|)+(|A_2|+|B_2|-1)+(|A_3|+|B_3|) =  |A|+|B|-1$$
and that
\[
|A'| \leq (3+2\times 100 \alpha^{-1}\beta^{-1})c_1' \leq c_1 \text{ and }|B'| \leq (3+2\times 100 \alpha^{-1}\beta^{-1})c_2' \leq c_2.
\]
This concludes the proof of Theorem~\ref{all_thm1}.
\end{proof}

\section{Open problems}

%
%

One very natural question to ask is as follows. Suppose that as usual we are choosing $c_1$ points of $A$ and $c_2$
points of $B$, where $|A|=|B|=n$ and $c_1c_2$ is a fixed multiple of $n$. Are there phenomena that may not occur in
the regime where say we are
choosing $c_1=n$ (in other words, we choose the whole of $A$) and $c_2$ bounded, but might possibly
always hold when both $c_1$ and $c_2$ are of order $\sqrt{n}$?
One example is the following.

\vspace{5pt}
{\bf Question 1. }{\em
  Is there a constant $c$ such that the following is true? If $A$ and $B$ are non-empty subsets of
${\mathbb Z}_p$ with $|A|=|B|=n \leq (p+1)/2$ then there are subsets $A'\subset A$ and
$B'\subset B$ with $|A'|=|B'|\le c \sqrt{n}$ such that $|A'+B'|\ge 2n-1$.
}

\vspace{5pt}

As remarked in the Introduction, this is not true for $c_1=n$ and $c_2$ bounded, in other words for $A'=A$ and
$|B'|$ bounded, as may be seen by taking $A$ and $B$ to be random subsets of ${\mathbb Z}_p$ of size approaching
$p/2$. But it might conceivably be true when we force both $c_1$ and $c_2$ to be large.

\vspace{5pt}
In a similar vein, one might ask whether the case of both $c_1$ and $c_2$ being of order $\sqrt{n}$ is
in fact always the `best' case (where $A$ and $B$ are set of size $n$, say). Thus for Theorem~\ref{all_thm1}
we would be asking the following.

\vspace{5pt}
{\bf Question 2. }{\em
  Let $c>0$ and $c_1=c_1(n)$ be such that whenever $A$ and $B$ are subsets of ${\mathbb Z}_p$ with
  $|A|=|B|=n \leq p/3$ there exist subsets $A'\subset A$ and
  $B'\subset B$, with $|A'| \leq c_1$ and $|B'| \leq cn/c_1$, such that $|A'+B'|\ge 2n-1$. Does it follow that
  whenever $A$ and $B$ are subsets of ${\mathbb Z}_p$ with
  $|A|=|B|=n \leq p/3$ there exist subsets $A'\subset A$ and
  $B'\subset B$ with $|A'|=|B'|\le \sqrt{cn}$ such that $|A'+B'|\ge 2n-1$?
}

\vspace{5pt}

One could also ask what the `worst' case is: is it when $c_1=c$ and $c_2=n$? More generally, is there
`monotonicity' as $c_1$ varies from $c$ to $\sqrt{cn}$?

\vspace{5pt}
It would also be very interesting to obtain good bounds on the constants appearing in our various results.
For example, in Theorem~\ref{all_thm1}, what is the form of the dependence of $c$ on $\alpha$ and $\beta$? 

\vspace{5pt}
Finally, we consider what happens for discrete versions of the Brunn--Minkowski inequality. Green and
Tao~\cite{GreenTao}
showed
that, given a dimension $k$ and a constant $\varepsilon >0$, there exists $t$ such that if $A$ is a subset of
${\mathbb Z}^k$ of size $n$ that is not
contained inside $t$ parallel hyperplanes (intuitively, $A$ `does not look lower-dimensional'),
then $|A+A|  \geq (2^k - \varepsilon) n$. We wonder if the following might be true. Although this is a question
about $\mathbb{Z}$ rather than $\mathbb{Z}_p$, we feel that the methods in this paper are likely to be
relevant.

\vspace{5pt}
{\bf Question 3. }{\em
For a given dimension $k$, does  there exist a constant $c$  such that the
following holds? For any $\varepsilon>0$ there exists $t$ such that if $A$ is a subset of ${\mathbb Z}^k$ of
size $n$ that is not contained in
$t$ parallel hyperplanes, then there exists a subset $A'$ of $A$ of size at most $c \sqrt{n}$ such
that $|A'+A'|  \geq (2^k - \varepsilon) n$.
}

\end{document}